\documentclass{elsarticle}

\usepackage{xspace,float,amsmath,url}
\usepackage{enumerate}

\usepackage{amsfonts}
\usepackage{color,amsmath}
\usepackage{amsthm}
  \usepackage{latexsym,amssymb}
  \usepackage[T1]{fontenc}
  \usepackage[utf8]{inputenc}
  \usepackage{comment}
  \usepackage{listings}
\usepackage{lineno}

\newcommand{\exclude}[1]{}

\def\acc#1{\left\{ #1 \right\}}

\renewcommand{\le}{\leqslant}
\renewcommand{\ge}{\geqslant}

\newtheorem{theorem}{Theorem}[section]
\newtheorem{lemma}[theorem]{Lemma}

\begin{document}

\begin{frontmatter}

\title{Avoiding Conjugacy Classes on the 5-letter Alphabet}

\author[affiliation]{Golnaz Badkobeh\fnref{fn1}}
\address[affiliation]{Goldsmiths, University of London}
\ead{g.badkobeh@gold.ac.uk}
\fntext[fn1]{Golnaz Badkobeh is supported by the Leverhulme Trust on the Leverhulme Early Career Scheme.  }

\author[lirmm]{Pascal Ochem\corref{cor}\fnref{anr}}
\address[lirmm]{LIRMM, Universit\'e de Montpellier, CNRS, Montpellier, France}
\cortext[cor]{Corresponding author}
\ead{ochem@lirmm.fr}
\fntext[anr]{This work is supported by the ANR project CoCoGro (ANR-16-CE40-0005).}


\begin{abstract}
We construct an infinite word $w$ over the $5$-letter alphabet such that
for every factor $f$ of $w$ of length at least two, 
there exists a cyclic permutation of $f$ that is not a factor of $w$.
In other words, $w$ does not contain a non-trivial conjugacy class.
This proves the conjecture in Gamard et al. [TCS 2018]
\end{abstract}


\end{frontmatter}

\section{Introduction}
We consider infinite words over a finite alphabet that do not contain all the conjugates of the same word $w$,
with the necessary condition that $|w|\ge2$. A recent interest in such words appeared in the context of
pattern avoidance~\cite{circular} and of iterative algebras~\cite{BellMadill}.
Bell and Madill~\cite{BellMadill} obtained a pure morphic word with this property (and some additional properties) over the $12$-letter alphabet.
Gamard et al.~\cite{circular} independently obtained a morphic word over the $6$-letter alphabet.
They also conjectured that the alphabet size can be lowered to $5$, which is best possible.
In this paper, we prove this conjecture using a morphic word.

Together with the construction of a morphic binary word avoiding every conjugacy class of length at least $5$
and a morphic ternary word avoiding every conjugacy class of length at least $3$~\cite{circular},
this settles the topic of the smallest alphabet needed to avoid every conjugacy class of length at least $k$.

\section{Main result}
Let $\varepsilon$ denote the empty word.
We consider the morphic word $w_5=G(F^\omega(\texttt{0}))$ defined by the following morphisms.

\noindent
\begin{minipage}[b]{0.4\linewidth}
\centering
$$\begin{array}{l}
 F(\texttt{0})=\texttt{01},\\
 F(\texttt{1})=\texttt{2},\\
 F(\texttt{2})=\texttt{03},\\
 F(\texttt{3})=\texttt{24},\\
 F(\texttt{4})=\texttt{23}.
\end{array}$$
\end{minipage}
\hspace{2mm}
\begin{minipage}[b]{0.4\linewidth}
\centering
$$\begin{array}{l}
 G(\texttt{0})=\texttt{abcd},\\
 G(\texttt{1})=\varepsilon,\\
 G(\texttt{2})=\texttt{eacd},\\
 G(\texttt{3})=\texttt{becd},\\
 G(\texttt{4})=\texttt{be}. 
\end{array}$$
\end{minipage}

\begin{theorem}\label{main}
The morphic word $w_5\in\Sigma_5^*$ avoids every conjugacy class of length at least 2.
\end{theorem}

In order to prove this theorem, it is convenient to express $w_5$ with the larger morphisms
$f=F^3$ and $g=G\circ F^2$ given below. Clearly, $w_5=g(f^\omega(\texttt{0}))$.

\noindent
\begin{minipage}[b]{0.4\linewidth}
\centering
$$\begin{array}{l}
 f(\texttt{0})=\texttt{01203},\\
 f(\texttt{1})=\texttt{0124},\\
 f(\texttt{2})=\texttt{0120323},\\
 f(\texttt{3})=\texttt{01240324},\\
 f(\texttt{4})=\texttt{01240323}.
\end{array}$$
\end{minipage}
\hspace{2mm}
\begin{minipage}[b]{0.4\linewidth}
\centering
$$\begin{array}{l}
 g(\texttt{0})=\texttt{abcdeacd},\\
 g(\texttt{1})=\texttt{abcdbecd},\\
 g(\texttt{2})=\texttt{abcdeacdbe},\\
 g(\texttt{3})=\texttt{abcdbecdeacdbecd},\\
 g(\texttt{4})=\texttt{abcdbecdeacdbe}. 
\end{array}$$
\end{minipage}


\subsection{Avoiding conjugacy classes in $F^\omega(\texttt{0})$}
Here we study the pure morphic word and the conjugacy classes it contains.


\begin{lemma} The infinite word $F^\omega(\texttt{0})$ contains only the conjugacy classes listed in
$C=\acc{F(\texttt{2}),F^2(\texttt{2}),F^d(\texttt{4}),f^d(\texttt{0})}$, for all $d\ge1$.
\end{lemma}

\begin{proof}
Notice that the factor \texttt{01} only occurs as the prefix of the $f$-image of every letter in $F^\omega(\texttt{0})$.
Moreover, every letter \texttt{1} only occurs in $F^\omega(\texttt{0})$ as the suffix of the factor \texttt{01}.
Let us say that the \emph{index} of a conjugacy class is the number of occurrences of \texttt{1} in any of its elements.
An easy computation shows that the set of complete conjugacy classes in $F^\omega(\texttt{0})$
with index at most one is $C_1=\acc{F(\texttt{2}),F^2(\texttt{2}),F(\texttt{4}),F^2(\texttt{4}),f(\texttt{4}),f(\texttt{0})}$.
Let us assume that $F^\omega(\texttt{0})$ contains a conjugacy class $c$ with index at least two.
Let $w\in c$ be such that \texttt{01} is a prefix of $w$.
We write $w=ps$ such that the leftmost occurrence of $\texttt{01}$ in $w$ is the prefix of $s$.
Then the conjugate $sp$ of $w$ also belongs to $c$ and thus is a factor of $F^\omega(\texttt{0})$.
This implies that the pre-image $v=f^{-1}(w)$ is a factor of $F^\omega(\texttt{0})$, and so does every conjugate of $v$.
Thus, $F^\omega(\texttt{0})$ contains a conjugacy class $c'$ such that the elements of $c$ with prefix $\texttt{01}$ are the $f$-images
of the elements of $c'$. Moreover, the index of $c'$ is strictly smaller than the index of $c$.

%
Using this argument recursively, we conclude that every complete conjugacy class in $F^\omega(\texttt{0})$
has a member of the form $f^i(x)$ such that $x$ is an element of a conjugacy class in $C_1$.

Now we show that $F(\texttt{2})$ does not generate larger conjugacy classes in $F^\omega(\texttt{0})$.
We thus have to exhibit a conjugate of $f(F(\texttt{2}))=F^4(\texttt{2})=\texttt{0120301240324}$ that is
not a factor of $F^\omega(\texttt{0})$. A computer check shows that the conjugate $\texttt{4012030124032}$
is not a factor of $F^\omega(\texttt{0})$.
Similarly, $F^2(\texttt{2})$ does not generate larger conjugacy classes in $F^\omega(\texttt{0})$
since the conjugate $\texttt{301203012401203230124032}$ of $f(F^2(\texttt{2}))=F^5(\texttt{2})=\texttt{012030124012032301240323}$
is not a factor of $F^\omega(\texttt{0})$.

\end{proof}
\subsection{Avoiding conjugacy classes in $w_5$}

We are ready to prove Theorem~\ref{main}.
Notice that $\texttt{ab}$ only appears in $w_5$ as the prefix of the $g$-image of every letter.
A computer check shows that $w_5$ avoids every conjugacy class of length at most $100$.
Consider a word $w$ with length at least $101$ whose conjugacy class is complete in $w_5$.
So $w$ contains at least two occurrences of $\texttt{ab}$.

Let us assume that $w_5$ contains a conjugacy class $c$.
Let $w\in c$ be such that $\texttt{ab}$ is a prefix of $w$.
Since $|w|\ge101$, $w$ contains at least $2$ occurrences of $\texttt{ab}$ and
we write $w=ps$ such that the leftmost occurrence of $\texttt{ab}$ in $w$ is the prefix of $s$.
Then the conjugate $sp$ of $w$ also belongs to $c$ and thus is a factor of $w_5$.
This implies that the pre-image $v=g^{-1}(w)$ is a factor of $F^\omega(\texttt{0})$, and so does every conjugate of $v$.
Thus, $F^\omega(\texttt{0})$ contains a conjugacy class $c'$ such that the elements of $c$ with prefix $\texttt{ab}$ are the $f$-images
of the elements of $c'$.

To finish the proof, it is thus sufficient to show that for every $c'\in C$, there exists a conjugate of $g(c')$ that is not a factor of $w_5$.
Recall that $C=\acc{F(\texttt{2}),F^2(\texttt{2}),F^d(\texttt{4}),f^d(\texttt{0})}$ for all $d\ge1$.
The computer check mentioned above settles the case of $F(\texttt{2})$ and $F^2(\texttt{2})$ since $|g(F(\texttt{2}))|<|g(F^2(\texttt{2}))|=40\le100$.

The next four lemmas handle the remaining cases:
\begin{itemize}
 \item $g(f^d(F(\texttt{4})))=g(f^d(\texttt{23}))$
 \item $g(f^d(F^2(\texttt{4})))=g(f^d(\texttt{0324}))$
 \item $g(f^{d+1}(\texttt{4}))=g(f^d(\texttt{01240323}))$
 \item $g(f^{d+1}(\texttt{0}))=g(f^d(\texttt{01203}))$
\end{itemize}

\begin{lemma}
Let $p_{\texttt{23}}=\texttt{e}.g(\texttt{3}f(\texttt{3})\ldots f^{d-1}(\texttt{3}).f^d(\texttt{3}))$ and\\
$s_{\texttt{23}}=g(f^{d-1}(\texttt{01203}).f^{d-2}(\texttt{01203})\ldots f(\texttt{01203})\texttt{01203}).\texttt{abcdeacdb}$.
For every $d\ge0$, the word $T_{\texttt{23}}=p_{\texttt{23}}s_{\texttt{23}}$ is a conjugate of $g(f^d(\texttt{23}))$ that is not a factor of $w_5$.
\end{lemma}
\begin{proof}
Let us assume that $T_{\texttt{23}}$ appears in $w_5$.\\
\noindent
The letter \texttt{3} in $f^\omega(\texttt{0})$ appears after either \texttt{0} or \texttt{2}.
However \texttt{e} is a suffix of $g(\texttt{2})$ and not of $g(\texttt{0})$. 
Therefore, $\texttt{e}.g(\texttt{3})$ is a suffix of $g(\texttt{23})$ only.
Since \texttt{23} is a suffix of $f(\texttt{2})$ and not of $f(\texttt{0})$, then $g(\texttt{23}f(\texttt{3}))$ is a suffix of $g(f(\texttt{23}))$ only.
Using this argument recursively, $p_{\texttt{23}}$ is a suffix of $g(f^d(\texttt{23}))$ only.

Now, the letter \texttt{3} in $f^\omega(\texttt{0})$ appears before either \texttt{0} or \texttt{2},
however \texttt{abcdeacdb} is a prefix of $g(\texttt{2})$ and not of $g(\texttt{0})$.
Thus $g(\texttt{01203}).\texttt{abcdeacdb}$ is a prefix of $g(\texttt{012032})$ only.
Since \texttt{012032} is a prefix of $f(\texttt{2})$ and not of $f(\texttt{0})$, then $g(f(\texttt{01203})\texttt{012032})$ is a prefix of $g(f(\texttt{012032}))$ only.
Using this argument recursively, $s_{\texttt{23}}$ is a prefix of $g(f^{d-1}(\texttt{012032}))$ only.
Thus $T_{\texttt{23}}$ is a factor of $g(f^d(\texttt{232}))$ but \texttt{232} is not a factor of $f^\omega(\texttt{0})$, contradiction.
\end{proof}
\begin{lemma}
Let $p_{\texttt{0324}}=\texttt{acdbecd}.g(\texttt{24}f(\texttt{24})\ldots f^{d-1}(\texttt{24})).f^d(\texttt{24}))$
and $s_{\texttt{0324}}= g(f^d(\texttt{0}).f^{d-1}(\texttt{01240})\ldots f(\texttt{01240}).\texttt{01240}).\texttt{abcdbecde}$.
For every $d\ge0$, the word $T_{\texttt{0324}}=p_{\texttt{0324}}g(f^d(\texttt{0}))s_{\texttt{0324}}$ is a conjugate of $g(f^d(\texttt{0324}))$ that is not a factor of $w_5$.
\end{lemma}
\begin{proof} 
Let us assume that $T_{\texttt{0324}}$ appears in $w_5$.\\
\noindent
The letter \texttt{2} in $f^\omega(\texttt{0})$ appears after either \texttt{1} or \texttt{3}.
However \texttt{acdbecd} is a suffix of $g(\texttt{3})$ and not of $g(\texttt{1})$. 
Therefore $\texttt{acdbecd}.g(\texttt{24})$ is a suffix of $g(\texttt{324})$ only.
Since \texttt{324} is a suffix of $f(\texttt{3})$ and not of $f(\texttt{1})$, then $g(\texttt{324}f(\texttt{24}))$ is a suffix of $g(f(\texttt{324}))$ only.
Using this argument recursively, $p_{\texttt{3240}}$ is a suffix of $g(f^d(\texttt{324}))$ only.

Now, the letter \texttt{0} in $f^\omega(\texttt{0})$ appears before either \texttt{1} or \texttt{3}.
However \texttt{abcdbecde} is a prefix of $g(\texttt{3})$ and not of $g(\texttt{1})$.
Thus $g(\texttt{01240}).\texttt{abcdbecde}$ is a prefix of $g(\texttt{012403})$ only.
Since \texttt{012403} is a prefix of $f(\texttt{3})$ and not of $f(\texttt{1})$, then $g(f(\texttt{01240})\texttt{012403})$ is a prefix of $g(f(\texttt{012403}))$ only.
Using this argument recursively, $s_{\texttt{3240}}$ is a prefix of $g(f^{d-1}(\texttt{012403}))$ only.
Thus $T_{\texttt{3240}}$ is a factor of $g(f^d(\texttt{32403}))$ but \texttt{32403} is not a factor of $f^\omega(\texttt{0})$, contradiction.
\end{proof}
\begin{lemma}
Let $p_{\texttt{01240323}}=\texttt{ecdeacdbe}.g(\texttt{0323}f(\texttt{0323})\cdots f^{d-1}(\texttt{0323}).f^d(\texttt{0323}))$ and
$s_{\texttt{01240323}}=g(f^d(\texttt{012})f^{d-1}(\texttt{012}). \cdots f(\texttt{012})\texttt{012}).\texttt{abcdb}$.
For every $d\ge0$, the word $T_{\texttt{01240323}}=p_{\texttt{01240323}}s_{\texttt{01240323}}$ is a conjugate of $g(f^d(\texttt{0323})$ that is not a factor of $w_5$.
\end{lemma}
\begin{proof}
Let us assume that $T_{\texttt{01240323}}$ appears in $w_5$.\\
\noindent
The factor \texttt{03} in $f^\omega(\texttt{0})$ appears after either \texttt{2} or \texttt{4}.
However \texttt{ecdeacdbe} is a suffix of $g(\texttt{4})$ and not of $g(\texttt{2})$. 
Therefore $\texttt{ecdeacdbe}.g(\texttt{0323})$ is a suffix of $g(\texttt{40323})$ only.
Since \texttt{40323} is a suffix of $f(\texttt{4})$ and not of $f(\texttt{2})$, then $g(\texttt{40323}f(\texttt{0323}))$ is a suffix of $g(f(\texttt{40323}))$,
using this argument recursively, $p_{\texttt{01240323}}$ is a suffix of $g(f^d(\texttt{40323}))$ only.

Now, the factor \texttt{12} in $f^\omega(\texttt{0})$ appears before either \texttt{0} or \texttt{4}.
However \texttt{abcdb} is a prefix of $g(\texttt{4})$ and not of $g(\texttt{0})$.
Thus $g(\texttt{012}).\texttt{abcdb}$ must only be a prefix of $g(\texttt{0124})$ and since \texttt{0323} is a prefix of $f(\texttt{4})$
and not of $f(\texttt{0})$ then $g(f(\texttt{012})\texttt{0124})$ is a prefix of $g(f(\texttt{0124}))$ only.
Using this argument recursively, $s_{\texttt{01240323}}$ is a prefix of $g(f^d(\texttt{0124}))$ only.
Thus $T_{\texttt{01240323}}$ is a factor of $g(f^d(\texttt{403230124}))$ but \texttt{403230124} is not a factor of $f^\omega(\texttt{0})$, contradiction.
\end{proof}
\begin{lemma}
Let $p_{\texttt{01203}}=\texttt{d}.g(\texttt{3}f(\texttt{3})\ldots f^{d-1}(\texttt{3}).f^d(\texttt{3}))$ and\\
$s_{\texttt{01203}}=g(f^d(\texttt{012})f^{d-1}(\texttt{012}).f^{d-2}(\texttt{012})\ldots f(\texttt{012})\texttt{012}).\texttt{abcdeac}$.
For every $d\ge0$, the word $T_{\texttt{01203}}=p_{\texttt{01203}}s_{\texttt{01203}}$ is a conjugate of $g(f^d(\texttt{01203}))$ that is not a factor of $w_5$.
\end{lemma}
\begin{proof}
Let us assume that $T_{\texttt{01203}}$ appears in $w_5$.\\
\noindent
The letter \texttt{3} in $f^\omega(\texttt{0})$ appears after either \texttt{0} or \texttt{2}.
however \texttt{d} is a suffix of $g(\texttt{0})$ and not of $g(\texttt{2})$. 
Therefore $\texttt{d}.g(\texttt{2})$ is a suffix of $g(\texttt{12})$ only.
Since \texttt{12} is a suffix of $f(\texttt{1})$ and not of $f(\texttt{3})$, then $g(\texttt{12}f(\texttt{2}))$ is a suffix of $g(f(\texttt{12}))$ only.
Using this argument recursively, $p_{\texttt{01203}}$ is a suffix of $g(f^d(\texttt{12}))$ only.

Now, \texttt{012} in $f^\omega(\texttt{0})$ appears before either \texttt{1} or \texttt{4},
however \texttt{abcdeac} is only a prefix of $g(\texttt{1})$ and not of $g(\texttt{4})$.
Thus $g(\texttt{012}).\texttt{abcdeac}$ is a prefix of $g(\texttt{0120})$ only.
Since \texttt{0120} is a prefix of $f(\texttt{1})$ and not of $f(\texttt{4})$, then $g(f(\texttt{012})\texttt{0120})$ is a prefix of $g(f(\texttt{0120}))$ only.
Using this argument recursively,  $s_{\texttt{01203}}$ is a prefix of $g(f^d(\texttt{0120}))$.
Thus $T_{\texttt{01203}}$ is a factor of $g(f^d(\texttt{030120}))$ but \texttt{030120} is not a factor of $f^\omega(\texttt{0})$, contradiction.
\end{proof}

\end{document}